\documentclass[12pt,twoside]{amsart}
\usepackage{amssymb}
\usepackage{amsmath}
\usepackage{xypic}

\title{A remark on Kov\'acs' vanishing theorem}
\author{Osamu Fujino} 
\date{2012/2/20, version 1.09}
\address{Department of Mathematics, Faculty of Science, 
Kyoto University, Kyoto 606-8502, Japan}
\email{fujino@math.kyoto-u.ac.jp}

\newcommand{\Exc}[0]{{\operatorname{Exc}}}
\newcommand{\Supp}[0]{{\operatorname{Supp}}}
\newtheorem{thm}{Theorem}

\theoremstyle{definition}
\newtheorem{rem}[thm]{Remark}
\newtheorem*{ack}{Acknowledgments} 
\begin{document}

\maketitle 

\begin{abstract}
We give an alternative proof of Kov\'acs' vanishing theorem. 
Our proof is based on the standard arguments of the 
minimal model theory. We do not need the notion of Du Bois pairs. 
We reduce Kov\'acs' vanishing theorem to the well-known relative 
Kawamata--Viehweg--Nadel vanishing theorem.  
\end{abstract}

The following theorem is the main theorem of 
this paper, which we call Kov\'acs' vanishing theorem. 

\begin{thm}[{cf.~\cite[Theorem 1.2]{kovacs}}]\label{main}
Let $(X, \Delta)$ be a log canonical 
pair and let $f:Y\to X$ be a proper birational morphism 
from a smooth variety $Y$ such that 
$\Exc (f)\cup \Supp f_*^{-1}\Delta$ is a simple 
normal crossing divisor on $Y$. 
In this situation, we can write 
$$
K_Y=f^*(K_X+\Delta)+\sum _i a_i E_i. 
$$ 
We put $E=\sum _{a_i=-1}E_i$. 
Then we have  
$$
R^if_*\mathcal O_Y(-E)=0
$$ 
for every $i>0$. 
\end{thm}
In this short paper, we reduce Kov\'acs' vanishing theorem to the 
well-known relative Kawamata--Viehweg--Nadel vanishing theorem 
by taking a dlt blow-up. Our proof makes Kov\'acs' vanishing theorem 
more accessible. 
From our viewpoint, Theorem \ref{main} is a variant of the relative 
Kawamata--Viehweg--Nadel vanishing theorem. 

Throughout this paper, we will work over an algebraically closed filed $k$ of characteristic 
zero and 
freely use the standard notation of the minimal model theory.  

\begin{rem}In \cite{kovacs}, Kov\'acs proved a rather general vanishing theorem 
for Du Bois pairs 
(cf.~\cite[Theorem 6.1]{kovacs}) and use it to derive Theorem \ref{main}. 
For the details, see \cite{kovacs}. 
\end{rem}

Before we give a proof of Theorem \ref{main}, we make a small remark. 

\begin{rem}
In \cite[Theorem 1.2]{kovacs}, $X$ is assumed to be 
$\mathbb Q$-factorial. 
Therefore, the statement of Theorem \ref{main} is slightly 
better than the original one (cf.~\cite[Theorem 1.2]{kovacs}). 
However, we can check that Theorem \ref{main} follows from \cite[Theorem 1.2]{kovacs}. 
\end{rem}

The following remark is important and seems to be well known to the experts. 

\begin{rem}\label{rem2}
The sheaf $R^if_*\mathcal O_Y(-E)$ is independent of the choice of  
$f:Y\to X$ for every $i$. 
It can be checked easily by the standard arguments based on the 
weak factorization theorem (cf.~\cite[Lemma 6.5.1]{kovacs}). 
For related topics, see \cite[Lemma 4.2]{fujino-lc}. 
\end{rem}

Let us start the proof of Theorem \ref{main}. It is essentially the 
same as the proof of \cite[Theorem 4.14]{book} 
(see also \cite[Proposition 2.4]{fujino-lc}). 

\begin{proof}[Proof of {\em{Theorem \ref{main}}}] 
By shrinking $X$, we may assume that 
$X$ is quasi-projective. 
We take a dlt blow-up $g:Z\to X$ (see, for example, \cite[Section 4]{ssmmp}). 
This means that $g$ is a projective birational morphism, 
$K_Z+\Delta_Z=g^*(K_X+\Delta)$, and 
$(Z, \Delta_Z)$ is a $\mathbb Q$-factorial 
dlt pair. 
By using Szab\'o's resolution lemma, we 
take a resolution of singularities $h:Y\to Z$ with 
the following properties. 
\begin{itemize}
\item[(1)] $\Exc (h)\cup \Supp h_*^{-1}\Delta_Z$ is a simple normal crossing 
divisor on $Y$. 
\item[(2)] $h$ is an isomorphism 
over the generic point of 
any lc center of $(Z, \Delta_Z)$. 
\end{itemize} 
We can write 
$$K_Y+h_*^{-1}\Delta_Z=h^*(K_Z+\Delta_Z)+F.$$ 
We put $f=g\circ h: Y\to X$. 
In this situation, $E=\llcorner h_*^{-1}\Delta_Z\lrcorner$. 
Note that 
$\ulcorner F\urcorner$ is effective and $h$-exceptional 
by the construction. We also note that 
$\Exc (f)\cup \Supp f_*^{-1}\Delta$ is not necessarily 
a simple normal crossing divisor on $Y$ in the above construction. 
We consider the following 
short exact sequence 
$$0\to \mathcal O_Y(-E+\ulcorner F\urcorner)\to 
\mathcal O_Y(\ulcorner F\urcorner)\to \mathcal O_{E}(\ulcorner 
F|_{E}\urcorner)\to 0. $$ 
Since 
$-E+F\sim _{\mathbb R, h}K_Y+\{h_*^{-1}\Delta_Z\}$ and 
$F\sim _{\mathbb R, h}K_Y+h_*^{-1}\Delta_Z$, 
we have $$-E+\ulcorner F\urcorner \sim _{\mathbb R, h}K_Y
+\{h_*^{-1}\Delta_Z\}+\{-F\}$$ and 
$$\ulcorner F\urcorner \sim _{\mathbb R, h} K_Y+h_*^{-1}\Delta_Z+\{-F\}. $$ 
By the relative Kawamata--Viehweg vanishing theorem and 
the vanishing theorem of Reid--Fukuda type (see, for example, 
\cite[Lemma 4.10]{book}), we have  
$$R^ih_*\mathcal O_Y(-E+\ulcorner F\urcorner)=R^ih_*\mathcal O_Y(\ulcorner 
F\urcorner)=0$$ for every $i>0$. 
Therefore, we have a short exact sequence $$0\to 
h_*\mathcal O_Y(-E+\ulcorner F\urcorner)\to 
\mathcal O_Z\to h_*\mathcal O_{E}(\ulcorner 
F|_{E}\urcorner)\to 0$$ and 
$R^ih_*\mathcal O_{E}(\ulcorner F|_{E}\urcorner)=0$ 
for every $i>0$. Note that $\ulcorner F\urcorner $ is effective 
and $h$-exceptional. 
Thus we obtain 
$$\mathcal O_{\llcorner \Delta_Z\lrcorner }\simeq 
h_*\mathcal O_{E}\simeq h_*\mathcal O_{E} 
(\ulcorner F|_{E}\urcorner). $$ 
By the above vanishing result, 
we obtain $Rh_*\mathcal O_{E}(\ulcorner 
F|_{E}\urcorner)\simeq 
\mathcal O_{\llcorner \Delta_Z\lrcorner}$ in the derived category 
of coherent sheaves on $\llcorner \Delta_Z\lrcorner$. 
Therefore, the composition 
$$\mathcal O_{\llcorner \Delta_Z\lrcorner}\overset{\alpha}\longrightarrow 
R h_*\mathcal O_{E}\overset{\beta}\longrightarrow 
Rh_*\mathcal O_{E}(\ulcorner F|_{E}\urcorner)\simeq 
\mathcal O_{\llcorner \Delta_Z\lrcorner}$$ is 
a quasi-isomorphism. 
Apply $R\mathcal Hom_{\llcorner \Delta_Z\lrcorner} 
(\underline{\ \ \ } ,\, \omega^{\bullet}_{\llcorner \Delta_Z\lrcorner})$ 
to 
$$\mathcal O_{\llcorner \Delta_Z\lrcorner}\overset {\alpha}\longrightarrow 
Rh_*\mathcal O_{E}\overset{\beta}\longrightarrow \mathcal O_{\llcorner 
\Delta_Z\lrcorner},$$ where 
$\omega_{\llcorner \Delta_Z\lrcorner}^{\bullet}$ is the dualizing complex 
of $\llcorner \Delta_Z\lrcorner$. 
Then we obtain that 
$$\omega^{\bullet}_{\llcorner \Delta_Z\lrcorner }\overset{a}\longrightarrow 
R h_*\omega^{\bullet}_{E} 
\overset{b}\longrightarrow 
\omega^{\bullet}_{\llcorner \Delta_Z\lrcorner}$$ 
and that $b\circ a$ is a quasi-isomorphism 
by the Grothendieck duality, where 
$\omega_E^{\bullet}\simeq \omega_E[\dim E]$ is the dualizing complex of $E$. 
Hence, we have  
$$h^i(\omega^{\bullet}_{\llcorner \Delta_Z\lrcorner})\subseteq R^ih_*\omega^{\bullet}_{E} 
\simeq R^{i+d}h_*\omega_{E},$$ where 
$d=\dim E=\dim \llcorner \Delta_Z\lrcorner =\dim X-1$. 
By the vanishing theorem (see, for 
example, \cite[Lemma 2.33]{book} and \cite[Lemma 3.2]{vanishing}), 
$R^ih_*\omega_{E}=0$ for 
every $i>0$. 
Therefore, $h^i(\omega^{\bullet}_{\llcorner \Delta_Z\lrcorner})=0$ for every $i>-d$. 
Thus, $\llcorner \Delta_Z\lrcorner$ is Cohen--Macaulay. 
This implies $\omega_{\llcorner \Delta_Z\lrcorner}^{\bullet}\simeq \omega_{\llcorner \Delta_Z\lrcorner}[d]$. 
Since $E$ is a simple normal crossing divisor on $Y$ and $\omega_{E}$ is an invertible 
sheaf on $E$, every associated prime of $\omega_{E}$ is the generic point of some irreducible 
component of $E$. By $h$, 
every irreducible component of $E$ is mapped birationally onto 
an irreducible component of $\llcorner \Delta_Z\lrcorner$. 
Therefore, $h_*\omega_{E}$ is a pure sheaf on $\llcorner \Delta_Z\lrcorner$. 
Since the composition 
$$\omega_{\llcorner \Delta_Z\lrcorner}
\to h_*\omega_{E}\to \omega_{\llcorner \Delta_Z\lrcorner}$$ 
is an isomorphism, which is induced by $a$ and $b$ above, we obtain 
$h_*\omega_{E}\simeq \omega_{\llcorner \Delta_Z\lrcorner}$. 
It is because $h_*\omega_{E}$ is generically isomorphic to $\omega_{\llcorner \Delta_Z\lrcorner}$. 
By the Grothendieck duality, 
\begin{align*}
Rh_*\mathcal O_{E}&\simeq 
R\mathcal Hom _{\llcorner \Delta_Z\lrcorner}(Rh_*\omega^{\bullet}_{E}, 
\, \omega^{\bullet}_{\llcorner \Delta_Z\lrcorner})\\ 
&\simeq 
R\mathcal Hom _{\llcorner \Delta_Z\lrcorner}(\omega^{\bullet}_{\llcorner 
\Delta_Z\lrcorner},\,  
\omega^{\bullet}_{\llcorner \Delta_Z\lrcorner})\simeq \mathcal O_{\llcorner \Delta_Z\lrcorner} 
\end{align*} 
in the derived category of coherent sheaves on $\llcorner \Delta_Z\lrcorner$. 
In particular, $R^ih_*\mathcal O_{E}=0$ for every $i>0$. 
Since 
$Z$ has only rational singularities, 
we have $R^ih_*\mathcal O_Y=0$ for every $i>0$ and 
$h_*\mathcal O_Y\simeq \mathcal O_Z$. 
Thus, we can easily check that 
$R^ih_*\mathcal O_Y(-E)=0$ for every $i>0$ 
by using the exact sequence 
$$
0\to \mathcal O_Y(-E)\to \mathcal O_Y\to \mathcal O_E\to 0. 
$$
Note that $h_*\mathcal O_E\simeq \mathcal O_{\llcorner \Delta_Z\lrcorner}$. 
We can also check that $h_*\mathcal O_Y(-E)=\mathcal J(Z, \Delta_Z)$, where 
$\mathcal J(Z, \Delta_Z)$ is the 
multiplier ideal sheaf associated to the pair $(Z, \Delta_Z)$. 
Note that $\mathcal J(Z, \Delta_Z)=\mathcal O_Z(-\llcorner \Delta_Z\lrcorner)$ in 
our situation. 
Therefore, 
$$R^if_*\mathcal O_Y(-E)\simeq R^ig_*\mathcal J(Z, \Delta_Z)$$ for 
every $i$ by Leray's spectral sequence. 
By the relative Kawamata--Viehweg--Nadel vanishing theorem, 
$R^ig_*\mathcal J(Z, \Delta_Z)=0$ for every $i>0$. 
Thus we obtain $R^if_*\mathcal O_Y(-E)=0$ for every $i>0$. 
Note that $\Exc (f)\cup \Supp f_*^{-1}\Delta$ is not necessarily a simple normal crossing 
divisor on $Y$ in the above construction. 
Let $\mathcal I_{\Exc (f)}$ be the defining ideal sheaf of 
$\Exc (f)$ on $Y$. 
Apply the principalization of $\mathcal I_{\Exc(f)}$. 
Then we obtain a sequence of blow-ups whose centers have simple 
normal crossings with $\Exc (h)\cup \Supp h_*^{-1}\Delta_Z$ 
(see, for example, \cite[Theorem 3.35]{kollar}). 
In this process, 
$R^if_*\mathcal O_Y(-E)$ does not change for every $i$ 
as in Remark \ref{rem2} (see also \cite[4.6]{fujino-lc}). 
Therefore, we may assume that 
$\Exc (f)\cup \Supp f_*^{-1}\Delta$ is a simple normal crossing 
divisor on $Y$. 
Remark \ref{rem2} completes the proof of Theorem \ref{main}. 
\end{proof}

\begin{ack}
The author was partially supported by the Grant-in-Aid for Young Scientists 
(A) $\sharp$20684001 from JSPS. 
\end{ack}

\end{document}